\documentclass[12pt,reqno]{amsart}

\usepackage{fancyhdr}
\newtheorem{theorem}{Theorem}[section]

\usepackage{graphics}
\usepackage{amssymb}
\usepackage{cite}
\usepackage{mathrsfs}
\usepackage{amsmath}

\newtheorem{remark}[theorem]{Remark}

\newtheorem{problem}[theorem]{Problem}

\numberwithin{equation}{section}

\large\normalsize
\title{}

\begin{document}
\begin{center}
{\Large\bf On the edge reconstruction of the characteristic and permanental polynomials of a simple graph}
\\[10pt]

{Jingyuan\ Zhang \footnote{Email address: doriazhang@outlook.com.}}
{Xian'an \ Jin \footnote{Partially supported by NSFC Grant (12171402); Email address: xajin@xmu.edu.cn.}}
{Weigen \ Yan$^*$ \footnote{Corresponding author; Partially supported by NSFC Grant (12071180); Email address: weigenyan@263.net.}}
{Qinghai \ Liu \footnote{Partially supported by NSFC Grant (11871015); Email address: qliu@fzu.edu.cn.}}
\\[10pt]
\footnotesize{$^{1,2}$ School of Mathematical Sciences, Xiamen University, Xiamen 361005, China.}\\
\footnotesize{$^{3}$ School of Science, Jimei University, Xiamen 361021, China}\\
\footnotesize{$^{4}$ College of Mathematics and Computer Science, Fuzhou University, Fuzhou 350108, China.}
\end{center}
\begin{abstract}
As a variant of the Ulam's vertex reconstruction conjecture and the Harary's edge reconstruction conjecture, Cvetkovi\'c and Schwenk posed
independently the following problem: Can the characteristic polynomial of a simple graph $G$ with vertex set $V$ be reconstructed from the
characteristic polynomials of all subgraphs in $\{G-v|v\in V\}$ for $|V|\geq 3$? This problem is still open.  A natural problem is: Can the
characteristic polynomial of a simple graph $G$ with edge set $E$  be reconstructed from the characteristic polynomials of all subgraphs in
$\{G-e|e\in E\}$? In this paper, we prove that if $|V|\neq |E|$, then the characteristic polynomial of $G$ can be reconstructed from the
characteristic polynomials of all subgraphs in $\{G-uv, G-u-v|uv\in E\}$, and the similar result holds for the permanental polynomial of $G$.
We also prove that the Laplacian (resp. signless Laplacian) characteristic polynomial of $G$ can be reconstructed
from the Laplacian (resp. signless Laplacian) characteristic polynomials of all subgraphs in $\{G-e|e\in E\}$ (resp. if $|V|\neq |E|$).
\end{abstract}

\maketitle

\section{Introduction}
\hspace*{\parindent}
The famous Ulam's vertex reconstruction conjecture \cite{Ulam60} asserts that
each simple graph $G$ with at least three vertices can be uniquely reconstructed from its vertex deck $\{G-v|v\in V(G)\}$.
Harary \cite{Har65,HM69} posed a similar conjecture (i.e., the edge reconstruction conjecture), which states that every simple graph $G$ with
edge set $E(G)$ can be reconstructed from its edge deck $\{G-e|e\in E(G)\}$ if $|E(G)|\geq 4$.
Although many results for these two conjectures have been obtained
\cite{And82,For04,Godsil87,GM81,Hos22,KNWZ21,Lov72,Man76,Mul77,Ram81,Yuan82,Thom77}, they are still open.

At the XVIII International Scientific Colloquium in Ilmenau in 1973, Cvetkovi\'c posed a related problem as follows: Can the characteristic
polynomial $\sigma(G;x)$ of a simple graph $G$ with vertex set $V(G)$ be reconstructed from the characteristic polynomials of subgraphs in
$\{G-v|v\in V(G)\}$ for $|V(G)|\geq 3$? The same problem was independently posed by Schwenk \cite{Sch79}.

Note that $\sigma'(G;x)=\sum\limits_{v\in V(G)}\sigma(G-v;x)$. Hence the coefficients of $\sigma(G;x)$ except for the constant term can be
reconstructed from $\{\sigma(G-v;x)|v\in V(G)\}$. Gutman and Cvetkovi\'c \cite{GC75} prove that the constant term of $\sigma(G;x)$ can be
uniquely determined from $\{\sigma(G-v;x)|v\in V(G)\}$ for all trees $G$ except perhaps for trees having a 1-factor. No examples of non-unique
reconstruction of the characteristic polynomial of graphs are known. Under the assumption that the reconstruction of the characteristic
polynomial is not unique, Cvetkovi\'c \cite{Cvet00} described some properties of graphs $G$ such that the constant term of $\sigma(G;x)$ can
not be reconstructed from $\{\sigma(G-v;x)|v\in V(G)\}$. Hagos \cite{Hag00} proved that the characteristic polynomial of a graph $G$ is
reconstructible from the characteristic polynomials of its vertex-deleted subgraphs and their complements. For the reconstruction problem of
the characteristic polynomial of graphs, see for example the nice survey \cite{SS23} and papers \cite{BIK05,SS16}. Recently, we considered the
edge reconstruction problem of the six digraph polynomials \cite{ZJY23}.

We assume that $G=(V(G),E(G))$ is a simple graph with vertex set $V(G)=\{v_1,v_2,\ldots,\\v_n\}$ and edge set $E(G)=\{e_1,e_2,\ldots,e_m\}$.
The diagonal matrix of vertex degrees of $G$ is denoted by $D=diag(d_1,d_2,\ldots, d_n)$, where $d_i$ is the degree of the vertex $v_i$ in $G$.
The adjacency matrix of $G$ is an $n\times n$ matrix $A$ whose entries $a_{ij}$ are given by
\begin{equation*}
a_{ij}=\left\{
\begin{array}{ll}
1, & \mbox{if}\ v_i\ \mbox{and}\ v_j\ \mbox{are adjacent}; \\
0, & \mbox{otherwise}.
\end{array}
\right.
\end{equation*}
Then $D-A$ and $D+A$ are the Laplacian and the signless Laplacian matrices of $G$, respectively. The collection $\{G-v_sv_t,
G-v_s-v_t|v_sv_t\in E(G)\}$ of subgraphs of $G$ is called the edge-vertex deck of $G$.
Set
\begin{align}
\sigma_1(G;x)&=\det(xI-A),\\
\sigma_2(G;x)&=\det(xI-D+A),\\
\sigma_3(G;x)&=\det(xI-D-A),\\
\sigma_4(G;x)&={\rm per}(xI-A),
\end{align}
where $\det(X)$ and ${\rm per}(X)$ denote the determinant and permanent of a square matrix $X$, respectively. Then $\sigma_1(G;x),
\sigma_2(G;x),\sigma_3(G;x)$ and $\sigma_4(G;x)$ are called the characteristic polynomial, Laplacian characteristic polynomial, signless
Laplacian characteristic polynomial, permanental polynomial of $G$, respectively.

A similar problem to the reconstruction of the characteristic polynomial of a graph $G$ from the characteristic polynomials of the deck of $G$
is the following:
\begin{problem}
For any $i=1,2,3,4$, can the graph polynomial $\sigma_i(G;x)$ of a graph $G$ be reconstructed from $\{\sigma_i(G-e;x)|e\in E(G)\}$?
\end{problem}

In the next section, we obtain two identities related to the determinants and the permanents, respectively. Using these two identities, in
Section 3,  we obtain the following results:
\begin{equation*}
(m-n)\sigma_1(G;x)+x\sigma_1'(G;x)=\sum_{v_sv_t\in E(G)}[\sigma_1(G-v_sv_t;x)+\sigma_1(G-v_s-v_t;x)],
\end{equation*}
\begin{equation*}
(m-n)\sigma_4(G;x)+x\sigma_4'(G;x)=\sum_{v_sv_t\in E(G)}[\sigma_4(G-v_sv_t;x)-\sigma_4(G-v_s-v_t;x)],
\end{equation*}
and for any $i=2,3$,
\begin{equation*}
(m-n)\sigma_i(G;x)+x\sigma_i'(G;x)=\sum_{e\in E(G)}\sigma_i(G-e;x).
\end{equation*}

In Section 4, we prove that, $\sigma_i(G;x)$ can be reconstructed from $\{\sigma_i(G-v_sv_t;x), \sigma_i(G-v_s-v_t;x)|v_sv_t\in E(G)\}$ for any
$i=1,4$ if $m\neq n$, $\sigma_2(G;x)$ can be reconstructed from $\{\sigma_2(G-e;x)|e\in E(G)\}$, and $\sigma_3(G;x)$ can be reconstructed from
$\{\sigma_3(G-e;x)|e\in E(G)\}$ if $m\neq n$.

\begin{remark}
Kiani and Mirzakhah \cite{KM15} used a different technique to obtain the results above on $\sigma_2(G;x)$ and $\sigma_3(G;x)$ for the mixed graphs.
\end{remark}

\section{Two identities}
Let $M$ be an $m\times n$ matrix. Given two sets $A=\{a_1,\ldots,a_i\}\subset\{1,2,\ldots,m\}$ and
$B=\{b_1,\ldots,b_j\}\subset\{1,2,\ldots,n\}$, let $M^A_B$ be a submatrix obtained from $M$ by deleting all the rows in $A$ and all the columns
in $B$. For convenience, we write $M^{a_1,\ldots,a_i}_{b_1,\ldots,b_j}$ instead of $M^{\{a_1,\ldots,a_i\}}_{\{b_1,\ldots,b_j\}}$.

Let $X=(x_{st})_{n\times n}$ be a symmetric matrix of order $n$ over the complex field. Hence $x_{st}=x_{ts}$ for any $1\leq s,t\leq n$. For
any $1\leq i, j\leq n$, define a symmetric matrix $X_{[ij]}=(x_{st}^{ij})_{n\times n}$, where
\begin{displaymath}
x_{st}^{ij} = \left\{ \begin{array}{ll}
x_{st}, & \textrm{if}\ (s,t)\neq (i,j)\ \textrm{and}\ (s,t)\neq (j,i); \\
0, & \textrm{otherwise}.
\end{array} \right.
\end{displaymath}
That is, $X_{[ij]}=X_{[ji]}$ is the symmetric matrix obtained from $X$ by replacing the $(i,j)$-entry $x_{ij}$ and the $(j,i)$-entry $x_{ji}$
with $0$. For example, if
$$X=\left(\begin{array}{ccc}
x_{11} & x_{12} & x_{13}\\
x_{12} & x_{22} & x_{23}\\
x_{13} & x_{23} & x_{33}
\end{array}
\right),
$$
then
$$X_{[12]}=X_{[21]}=\left(\begin{array}{ccc}
x_{11} & 0 & x_{13}\\
0 & x_{22} & x_{23}\\
x_{13} & x_{23} & x_{33}
\end{array}
\right),
X_{[33]}=\left(\begin{array}{ccc}
x_{11} & x_{12} & x_{13}\\
x_{12} & x_{22} & x_{23}\\
x_{13} & x_{23} & 0
\end{array}
\right).
$$
Obviously, if $x_{ij}=0$, then $X=X_{[ij]}=X_{[ji]}$. Now we can prove the following result which will play an important role in the proof of
the main results in this paper.

\begin{theorem}
Let $X=(x_{st})_{n\times n}$ be a symmetric matrix of order $n$ over the complex field and let $X_{[ij]}$ be defined as above. Then the
determinant of $X$ satisfies:
\begin{equation}
\frac{1}{2}(n^2-n)\det(X)=\sum_{1\leq i\leq j\leq n}\det(X_{[ij]})+\sum_{1\leq i<j\leq n}x_{ij}^2\det(X^{i,j}_{i,j}).
\end{equation}
\end{theorem}
\begin{proof}
Note that, for any $1\leq s\leq n$ and $1\leq i\neq j\leq n$,
\begin{equation}
\det(X_{[ss]})=\det(X)-x_{ss}\det(X^s_s),
\end{equation}
\begin{equation}
\det(X_{[ij]})=\det(X)-(-1)^{i+j}x_{ij}\det(X^i_j)-(-1)^{i+j}x_{ji}\det(X^j_i)-x_{ij}^2\det(X^{i,j}_{i,j}).
\end{equation}

By Eqs. (2.2) and (2.3),
\begin{align}
&\sum_{1\leq i\leq j\leq n}\det(X_{[ij]})\nonumber\\
&=\sum_{i=1}^n\det(X_{[ii]})+\sum_{1\leq i<j\leq n}\det(X_{[ij]})\nonumber\\
&=\sum_{i=1}^n\det(X)-\sum_{i=1}^nx_{ii}\det(X^i_i)+\sum_{1\leq i<j\leq n}\det(X)-\sum_{\substack{1\leq i,j\leq n\\i\neq
j}}(-1)^{i+j}x_{ij}\det(X^i_j)\nonumber\\
&\ \ \ \ -\sum_{1\leq i<j\leq n}x_{ij}^2\det(X^{i,j}_{i,j})\nonumber
\end{align}
\begin{align}
&=n\cdot\det(X)-\sum_{i=1}^nx_{ii}\det(X^i_i)+\frac{n^2-n}{2}\cdot\det(X)-\sum_{\substack{1\leq i,j\leq n\\i\neq
j}}(-1)^{i+j}x_{ij}\det(X^i_j)\nonumber\\
&\ \ \ \ -\sum_{1\leq i<j\leq n}x_{ij}^2\det(X^{i,j}_{i,j})\nonumber\\
&=\frac{n^2+n}{2}\det(X)-\sum_{i=1}^n\sum_{j=1}^nx_{ij}(-1)^{i+j}\det(X_j^i)-\sum_{1\leq i<j\leq n}x_{ij}^2\det(X^{i,j}_{i,j})\nonumber\\
&=\frac{n^2+n}{2}\det(X)-n\cdot\det(X)-\sum_{1\leq i<j\leq n}x_{ij}^2\det(X^{i,j}_{i,j})\nonumber\\
&=\frac{n^2-n}{2}\det(X)-\sum_{1\leq i<j\leq n}x_{ij}^2\det(X^{i,j}_{i,j}).
\end{align}
Hence the theorem holds.
\end{proof}

Note that the permanent of a matrix $X=(x_{ij})_{n\times n}$ is defined as
\begin{equation*}
{\rm per}(X)=\sum_{\alpha\in S_n}x_{1\alpha(1)}x_{2\alpha(2)}\ldots x_{n\alpha(n)},
\end{equation*}
where $\alpha$ ranges over the set of the symmetric group of order $n$. Similarly, we obtain the following result.

\begin{theorem}
Let $X=(x_{st})_{n\times n}$ be a symmetric matrix of order $n$ over the complex field and let $X_{[ij]}$ be defined as above.  Then the
permanent ${\rm per}(X)$ of $X$ satisfies:
\begin{equation}
\frac{1}{2}(n^2-n){\rm per}(X)=\sum_{1\leq i\leq j\leq n}{\rm per}(X_{[ij]})-\sum_{1\leq i<j\leq n}x_{ij}^2{\rm per}(X^{i,j}_{i,j}).
\end{equation}
\end{theorem}
\begin{proof}
Note that
\begin{align}
&{\rm per}(X_{[ii]})={\rm per}(X)-x_{ii}{\rm per}(X^i_i),\\
&{\rm per}(X_{[ij]})={\rm per}(X)-x_{ij}{\rm per}(X^i_j)-x_{ji}{\rm per}(X^j_i)+x_{ij}^2{\rm per}(X^{i,j}_{i,j}),
\end{align}
for any $i,j\in\{1,2,\ldots,n\}$ and $i\neq j$.

By Eqs. (2.6) and (2.7),
\begin{align}
&\sum_{1\leq i\leq j\leq n}{\rm per}(X_{[ij]})\nonumber\\
&=\sum_{i=1}^n{\rm per}(X_{[ii]})+\sum_{1\leq i<j\leq n}{\rm per}(X_{[ij]})\nonumber\\
&=\sum_{i=1}^n{\rm per}(X)-\sum_{i=1}^nx_{ii}{\rm per}(X^i_i)+\sum_{1\leq i<j\leq n}{\rm per}(X)-\sum_{\substack{1\leq i,j\leq n\\i\neq
j}}x_{ij}{\rm per}(X^i_j)+\sum_{1\leq i<j\leq n}x_{ij}^2{\rm per}(X^{i,j}_{i,j})\nonumber\\
&=n\cdot{\rm per}(X)-\sum_{i=1}^nx_{ii}{\rm per}(X^i_i)+\frac{n^2-n}{2}\cdot{\rm per}(X)-\sum_{\substack{1\leq i,j\leq n\\i\neq j}}x_{ij}{\rm
per}(X^i_j)+\sum_{1\leq i<j\leq n}x_{ij}^2{\rm per}(X^{i,j}_{i,j})\nonumber\\
&=(n+\frac{n^2-n}{2}){\rm per}(X)-n\cdot{\rm per}(X)+\sum_{1\leq i<j\leq n}x_{ij}^2{\rm per}(X^{i,j}_{i,j})\nonumber\\
&=\frac{1}{2}(n^2-n){\rm per}(X)+\sum_{1\leq i<j\leq n}x_{ij}^2{\rm per}(X^{i,j}_{i,j}).
\end{align}
Hence the theorem holds.
\end{proof}

Let $2m$ be the number of non-zero non-diagonal entries of $X$, and let $2k$ be the number of zeros of non-diagonal entries in $X$. Then
$2m=n^2-n-2k$. The following result is equivalent to Theorem 2.1 above.

\begin{theorem}
Given a symmetric matrix $X=(x_{st})_{n\times n}$ of order $n$ over the complex field, let $2m$ be the number of non-zero non-diagonal entries
of $X$ and let $c$ be the number of zeros of diagonal entries in $X$. Then
\begin{equation}
(m-c)\det(X)=\sum_{(i,j)\in I_1}\det(X_{[ij]})+\sum_{(i,j)\in I_2}x_{ij}^2\det(X^{i,j}_{i,j}),
\end{equation}
where $I_1=\{(i,j)|x_{ij}\neq 0,\ 1\leq i\leq j\leq n\}$ and $I_2=\{(i,j)|x_{ij}\neq 0,\ 1\leq i<j\leq n\}$.
\end{theorem}
\begin{proof}
By Theorem 2.1,
\begin{align}
\frac{1}{2}(n^2-n)\det(X)&=\sum_{1\leq i\leq j\leq n}\det(X_{[ij]})+\sum_{1\leq i<j\leq n}x_{ij}^2\det(X^{i,j}_{i,j})\nonumber\\
&=\sum_{(i,j)\in I_1}\det(X_{[ij]})+\sum_{(i,j)\in I_2}x_{ij}^2\det(X^{i,j}_{i,j})+(c+
k)\det(X).
\end{align}
Hence the theorem holds.
\end{proof}

Similarly, by Theorem 2.2, we can prove the following result.

\begin{theorem}
Given a symmetric matrix $X=(x_{st})_{n\times n}$ of order $n$ over the complex field, let $2m$ be the number of non-zero non-diagonal entries
of $X$ and let $c$ be the number of zeros of diagonal entries in $X$. Then
\begin{equation}
(m-c){\rm per}(X)=\sum_{(i,j)\in I_1}{\rm per}(X_{[ij]})-\sum_{(i,j)\in I_2}x_{ij}^2{\rm per}(X^{i,j}_{i,j}),
\end{equation}
where $I_1=\{(i,j)|x_{ij}\neq 0,\ 1\leq i\leq j\leq n\}$ and $I_2=\{(i,j)|x_{ij}\neq 0,\ 1\leq i<j\leq n\}$.
\end{theorem}

\section{The generalized characteristic polynomial and generalized permanental polynomial }
Suppose that both $\beta$ and $\gamma$ are real numbers satisfying $\gamma\neq0$. Let $G$ be a simple graph with vertex set
$V(G)=\{v_1,v_2,\ldots,v_n\}$ and edge set $E(G)=\{e_1,e_2,\ldots,e_m\}$. Let $D$ and $A=(a_{ij})_{n\times n}$ denote the diagonal matrix of
vertex degrees and the adjacency matrix of $G$, respectively. The generalized characteristic polynomial and generalized permanental polynomial
of $G$ are defined as
\begin{equation}
\tau_1(G;x)=\det(xI_n-\beta D-\gamma A)
\end{equation}
and
\begin{equation}
\tau_2(G;x)={\rm per}(xI_n-\beta D-\gamma A),
\end{equation}
respectively.
\begin{theorem}
Let $\tau_1(G;x)=\det(xI_n-\beta D-\gamma A)$ and $\tau_2(G;x)={\rm per}(xI_n-\beta D-\gamma A)$ be the generalized characteristic polynomial
and generalized permanental polynomial of a graph $G$ with vertex set $V(G)=\{v_1,v_2,\ldots,v_n\}$ and edge set $E(G)=\{e_1,e_2,\ldots,e_m\}$.
Then
\begin{align}
&(m-n)\tau_1(G;x)+x\tau_1'(G;x)\nonumber\\
&=\sum_{e\in E(G)}\tau_1(G-e;x)+(\gamma^2-\beta^2)\sum_{v_sv_t\in E(G)}\det\left[(xI_n-\beta D-\gamma A)^{s,t}_{s,t}\right]\\
\nonumber
{\mbox{and}}&\\
\nonumber
&(m-n)\tau_2(G;x)+x\tau_2'(G;x)\nonumber\\
&=\sum_{e\in E(G)}\tau_2(G-e;x)-(\gamma^2+\beta^2)\sum_{v_sv_t\in E(G)}{\rm per}\left[(xI_n-\beta D-\gamma A)^{s,t}_{s,t}\right].
\end{align}
\end{theorem}
\begin{proof}
Note that $G$ has $m$ edges. Hence $xI_n-\beta D-\gamma A$ has $2m$ non-zero non-diagonal entries and contains no zeros of diagonal entries. By
Theorems 2.3 and $2.4$,
\begin{align}
&m\cdot\det(xI_n-\beta D-\gamma A)\nonumber\\
&=\sum_{i=1}^n \det(xI_n^{(i)}-\beta D^{(i)}-\gamma A)+\sum_{e\in E(G)}\det(xI_n-\beta D-\gamma A_e)\nonumber\\
&\ \ \ \ +\sum_{v_sv_t\in E(G)}(\gamma a_{st})^2\det[(xI_n-\beta D-\gamma A)^{s,t}_{s,t}],
&\\
\nonumber
&m\cdot{\rm per}(xI_n-\beta D-\gamma A)\nonumber\\
&=\sum_{i=1}^n {\rm per}(xI_n^{(i)}-\beta D^{(i)}-\gamma A)+\sum_{e\in E(G)}{\rm per}(xI_n-\beta D-\gamma A_e)\nonumber\\
&\ \ \ \ -\sum_{v_sv_t\in E(G)}(\gamma a_{st})^2{\rm per}[(xI_n-\beta D-\gamma A)^{s,t}_{s,t}],
\end{align}
where $A_e$ is the adjacency matrix of the graph $G-e$, and $I_n^{(i)}$ is the diagonal matrix of order $n$ with diagonal entries equal to one
except for the $i$-th entry equal to zero, and $D^{(i)}=diag(d_1,\ldots,d_{i-1},0,d_{i+1},\ldots,d_n)$.

Denote the diagonal matrix of vertex degrees of $G-e$ by $D_e$. Without loss of generality, let $e=v_sv_t\in E(G)$. Then
\begin{align*}
&\det(xI_n-\beta D-\gamma A_e)\nonumber\\
&=\det(xI_n-\beta D_e-\gamma A_e)-\beta\det[(xI_n-\beta D-\gamma A)^s_s]-\beta\det[(xI_n-\beta D-\gamma A)^t_t]\nonumber\\
&\ \ \ -\beta^2\det[(xI_n-\beta D-\gamma A)^{s,t}_{s,t}],
&\\
\nonumber\\
&{\rm per}(xI_n-\beta D-\gamma A_e)\nonumber\\
&={\rm per}(xI_n-\beta D_e-\gamma A_e)-\beta{\rm per}[(xI_n-\beta D-\gamma A)^s_s]-\beta{\rm per}[(xI_n-\beta D-\gamma A)^t_t]\nonumber\\
&\ \ \ -\beta^2{\rm per}[(xI_n-\beta D-\gamma A)^{s,t}_{s,t}].
\end{align*}
On the other hand,
\begin{align*}
&\det(xI_n^{(i)}-\beta D^{(i)}-\gamma A)=\det(xI_n-\beta D-\gamma A)-(x-\beta d_i)\det[(xI_n-\beta D-\gamma A)^i_i],\\
&{\rm per}(xI_n^{(i)}-\beta D^{(i)}-\gamma A)={\rm per}(xI_n-\beta D-\gamma A)-(x-\beta d_i){\rm per}[(xI_n-\beta D-\gamma A)^i_i].
\end{align*}
Hence
\begin{align}
&\sum_{e\in E(G)}\det(xI_n-\beta D-\gamma A_e)\nonumber\\
&=\sum_{e\in E(G)}\tau_1(G-e;x)-\beta\sum_{v_sv_t\in E(G)}\left[\det[(xI_n-\beta D-\gamma A)^s_s]+\det[(xI_n-\beta D-\gamma
A)^t_t]\right]\nonumber\\
&\ \ \ \ -\beta^2\sum_{v_sv_t\in E(G)}\det[(xI_n-\beta D-\gamma A)^{s,t}_{s,t}],
&\\
\nonumber\\
&\sum_{e\in E(G)}{\rm per}(xI_n-\beta D-\gamma A_e)\nonumber\\
&=\sum_{e\in E(G)}\tau_2(G-e;x)-\beta\sum_{v_sv_t\in E(G)}\left[{\rm per}[(xI_n-\beta D-\gamma A)^s_s]+{\rm per}[(xI_n-\beta D-\gamma
A)^t_t]\right]\nonumber\\
&\ \ \ -\beta^2\sum_{v_sv_t\in E(G)}{\rm per}[(xI_n-\beta D-\gamma A)^{s,t}_{s,t}],
&\\
\nonumber\\
&\sum_{i=1}^n\det(xI_n^{(i)}-\beta D^{(i)}-\gamma A)\nonumber\\
&=\sum_{i=1}^n\left[\det(xI_n-\beta D-\gamma A)-(x-\beta d_i)\det[(xI_n-\beta D-\gamma A)^i_i]\right]\nonumber\\
&=n\tau_1(G;x)-x\sum_{i=1}^n\det[(xI_n-\beta D-\gamma A)^i_i]+\beta\sum_{i=1}^nd_i\det[(xI_n-\beta D-\gamma A)_i^i],
&\\
\nonumber\\
&\sum_{i=1}^n{\rm per}(xI_n^{(i)}-\beta D^{(i)}-\gamma A)\nonumber\\
&=\sum_{i=1}^n\left[{\rm per}(xI_n-\beta D-\gamma A)-(x-\beta d_i){\rm per}[(xI_n-\beta D-\gamma A)^i_i]\right]\nonumber\\
&=n\tau_2(G;x)-x\sum_{i=1}^n{\rm per}[(xI_n-\beta D-\gamma A)^i_i]+\beta\sum_{i=1}^nd_i{\rm per}[(xI_n-\beta D-\gamma A)^i_i].
\end{align}

It is not difficult to show that
\begin{align}
&\sum_{i=1}^n\det[(xI_n-\beta D-\gamma A)^i_i]=\tau_1'(G;x),\\
&\sum_{i=1}^n{\rm per}[(xI_n-\beta D-\gamma A)^i_i]=\tau_2'(G;x).
\end{align}
By Eqs. (3.5)-(3.12),
\begin{align}
m\cdot\det(xI_n-\beta D-\gamma A)=&n\tau_1(G;x)-x\tau_1'(G;x)+\sum_{e\in E(G)}\tau_1(G-e;x)\nonumber\\
&+\sum_{v_sv_t\in E(G)}(\gamma^2a_{st}^2-\beta^2)\det[(xI_n-\beta D-\gamma A)^{s,t}_{s,t}],
\end{align}
\begin{align}
m\cdot{\rm per}(xI_n-\beta D-\gamma A)=&n\tau_2(G;x)-x\tau_2'(G;x)+\sum_{e\in E(G)}\tau_2(G-e;x)\nonumber\\
&-\sum_{v_sv_t\in E(G)}(\gamma^2a_{st}^2+\beta^2){\rm per}[(xI_n-\beta D-\gamma A)^{s,t}_{s,t}].
\end{align}
Thus the theorem holds.
\end{proof}

\begin{remark}
Similarly, for any edge-weighted graph $G$ with edge-weight function $\omega:E(G)\rightarrow\mathbb{R}\setminus\{0\}$, we have
\begin{align*}
&(m-n)\tau_1(G;x)+x\tau_1'(G;x)\nonumber\\
&=\sum_{e\in E(G)}\tau_1(G-e;x)+\sum_{v_sv_t\in E(G)}(\gamma^2\omega(v_sv_t)^2-\beta^2)\det\left[(xI_n-\beta D-\gamma A)^{s,t}_{s,t}\right].\\
\nonumber
{\mbox{and}}&\\
\nonumber
&(m-n)\tau_2(G;x)+x\tau_2'(G;x)\nonumber\\
&=\sum_{e\in E(G)}\tau_2(G-e;x)-\sum_{v_sv_t\in E(G)}(\gamma^2\omega(v_sv_t)^2+\beta^2){\rm per}\left[(xI_n-\beta D-\gamma
A)^{s,t}_{s,t}\right].
\end{align*}
\end{remark}

\begin{theorem}
Let $G$ be a simple graph with vertex set $V(G)=\{v_1,v_2,\ldots,v_n\}$ and edge set $E(G)=\{e_1,e_2,\ldots,e_m\}$. Then the graph polynomials
$\sigma_1(G;x)$ and $\sigma_4(G;x)$ defined in Section 1 satisfy:
\begin{equation}
(m-n)\sigma_1(G;x)+x\sigma_1'(G;x)=\sum_{v_sv_t\in E(G)}[\sigma_1(G-v_sv_t;x)+\sigma_1(G-v_s-v_t;x)],
\end{equation}
\begin{equation}
(m-n)\sigma_4(G;x)+x\sigma_4'(G;x)=\sum_{v_sv_t\in E(G)}[\sigma_4(G-v_sv_t;x)-\sigma_4(G-v_s-v_t;x)].
\end{equation}
\end{theorem}

\begin{proof}
It is not difficult to show that
\begin{align*}
&\det[(xI_n-A)^{i,j}_{i,j}]=\sigma_1(G-v_i-v_j;x),\\
&{\rm per}[(xI_n-A)^{i,j}_{i,j}]=\sigma_4(G-v_i-v_j;x).
\end{align*}
In Theorem 3.1, if we set $\beta=0$ and $\gamma=1$, then $\sigma_1(G;x)$ and $\sigma_4(G;x)$ satisfy Eqs. (3.15) and (3.16), respectively.
\end{proof}

If we set $\beta=1, \gamma=-1$ or $\beta=1, \gamma=1$ in Theorem 3.1, then $\gamma^2-\beta^2=0$. Hence we give a new proof of the following
result obtained in \cite{KM15}. In fact, they proved that the following Eq. (3.17) holds for any signed graph.
\begin{theorem}[\cite{KM15}]
Let $G$ be a simple graph with vertex set $V(G)=\{v_1,v_2,\ldots,v_n\}$ and edge set $E(G)=\{e_1,e_2,\ldots,e_m\}$. Then both $\sigma_2(G;x)$
and $\sigma_3(G;x)$ defined in Section 1 satisfy:
\begin{equation}
(m-n)\sigma_i(G;x)+x\sigma_i'(G;x)=\sum_{e\in E(G)}\sigma_i(G-e;x)
\end{equation}
for $i=2,3$.
\end{theorem}

\section{The edge reconstruction of $\sigma_1(G;.x),\sigma_2(G;x),\sigma_3(G;x)$ and $\sigma_4(G;x)$}

Using the results in Section 3, in this section, we discuss two problems as follows: (1) For $i=1,4$, can the graph polynomial $\sigma_i(G;x)$
defined in Section 1 be determined by $\{\sigma_i(G-e;x)|e\in E(G)\}\cup \{\sigma_i(G-v_s-v_t;x)|v_sv_t\in E(G)\}$? (2) For $i=2,3$, can the
graph polynomial $\sigma_i(G;x)$ defined in Section 1 be determined by $\{\sigma_i(G-e;x)|e\in E(G)\}$?

\begin{theorem}
Let $G$ be a simple graph with vertex set $V(G)=\{v_1,v_2,\ldots,v_n\}$ and edge set $E(G)=\{e_1,e_2,\ldots,e_m\}$. If $m\neq n$, then
$\sigma_i(G;x)$ can be reconstructed from $\{\sigma_i(G-e;x)|e\in E(G)\}\cup \{\sigma_i(G-v_s-v_t;x)|v_sv_t\in E(G)\}$ for $i=1,4$.
\end{theorem}
\begin{proof}
Note that, by Theorems 3.3, $\sigma_1(G;x)$ and $\sigma_4(G;x)$ satisfy the differential equations Eqs. (3.15) and (3.16), respectively. If
$m\neq n$, then
\begin{equation}
\sigma_1(G;0)=\frac{1}{m-n}\left[\sum_{e\in E(G)}\sigma_1(G-e;0)+\sum_{v_sv_t\in E(G)}\sigma_1(G-v_s-v_t;0)\right],
\end{equation}
\begin{equation}
\sigma_4(G;0)=\frac{1}{m-n}\left[\sum_{e\in E(G)}\sigma_4(G-e;0)-\sum_{v_sv_t\in E(G)}\sigma_4(G-v_s-v_t;0)\right].
\end{equation}
Hence both of the differential equations above have a unique solution, and the theorem holds.
\end{proof}

Kiani and Mirzakhah \cite{KM15} used a different method to prove the following theorem.

\begin{theorem}[\cite{KM15}]
Let $G$ be a simple graph with vertex set $V(G)=\{v_1,v_2,\ldots,v_n\}$ and edge set $E(G)=\{e_1,e_2,\ldots,e_m\}$. If $m\neq n$, then
$\sigma_i(G;x)$ can be reconstructed from $\{\sigma_i(G-e;x)|e\in E(G)\}$ for $i=2,3$.
\end{theorem}
\begin{proof}
Note that, for $i=2,3$, by Theorems 3.4, $\sigma_i(G;x)$ satisfies the differential equations Eq. (3.17). If $m\neq n$, then
$\sigma_i(G;0)=\frac{1}{m-n}\sum\limits_{e\in E(G)}\sigma_i(G-e;0)$.
Hence this differential equation has a unique solution, and the theorem holds.
\end{proof}

Note that, $\sigma_2(G;x)=\det(xI-D+A)$. Hence $\sigma_2(G;0)=\det(A-D)=0$. If $m=n$, then by Theorem 3.4,
$$x\sigma_2'(G;x)=\sum_{e\in E(G)}\sigma_2(G-e;x).$$
Given the initial condition $\sigma_2(G;0)=0$, the differential equation above has a unique solution. So, combining with Theorem 4.2, the
following theorem holds.

\begin{theorem}
The Laplacian  characteristic polynomial $\sigma_2(G;x)=\det(xI-D+A)$ of a simple graph $G$ with edge set $E(G)$ can be reconstructed from
$\{\sigma_2(G-e;x)|e\in E(G)\}$.
\end{theorem}

\section{Discussions}
In this paper, we solve mainly the following two problems: (1) Reconstruct the characteristic polynomial $\sigma_1(G;x)$ (the permanental
polynomial $\sigma_4(G;x)$) of a simple graph $G$ from the collection of characteristic polynomials (permanental polynomials) of the
edge-vertex deck of $G$ if $|V(G)|\neq |E(G)|$; (2) Reconstruct the Laplacian characteristic polynomial $\sigma_2(G;x)$ (the signless Laplacian
characteristic polynomial $\sigma_3(G;x)$) of a simple graph $G$ from the collection of Laplacian characteristic polynomials (signless
Laplacian characteristic polynomials) of the edge deck of $G$ (if $|V(G)|\neq |E(G)|$).

The following questions are interesting, which are still open.

1. Can the characteristic polynomial $\sigma_1(G;x)$ of a graph $G$ with $n$ vertices and $m$ edges be reconstructed from
$\{\sigma_1(G-e;x)|e\in E(G)\}$, or $\{\sigma_1(G-v_sv_t;x), \sigma_1(G-v_s-v_t)|v_sv_t\in
E(G)\}$ if $m=n$?

2. Can the permanental polynomial $\sigma_4(G;x)$ of a graph $G$ with $n$ vertices and $m$ edges be reconstructed from
$\{\sigma_4(G-e;x)|e\in E(G)\}$, or $\{\sigma_4(G-v_sv_t;x), \sigma_4(G-v_s-v_t)|v_sv_t\in
E(G)\}$ if $m=n$?

3. Can the sigless Laplacian characteristic polynomial $\sigma_3(G;x)$ be reconstructed from $\{\sigma_3(G-e;x)|e\in E(G)\}$ if $m=n$?

\end{document}